\documentclass[12 pt]{article}
\usepackage{geometry}
\usepackage{amsmath,amsfonts,amssymb, tikz, amsthm}
\usepackage{hyperref}
\usepackage{authblk}

\usepackage[utf8]{inputenc}
\usepackage{graphicx}
\usepackage[english]{babel}
\usepackage[tableposition=top]{caption} 
\usepackage{amsmath}
\usepackage{mathrsfs}
\usepackage{amssymb}
\usepackage{amsfonts}
\usepackage{mathtools}
\usepackage{tikz-cd}
\usepackage[utf8]{inputenc}

\newtheorem{theorem}{Theorem}[section]

\newtheorem{lemma}[theorem]{Lemma}
\newtheorem{proposition}[theorem]{Proposition}
\newtheorem{definition}[theorem]{Definition}

\theoremstyle{definition}
\newtheorem*{convention}{Notation}

\theoremstyle{remark}

\numberwithin{equation}{section}

\newcommand{\Z}{\mathbb{Z}}
\newcommand{\Q}{\mathbb{Q}}

\newcommand{\R}{\mathbb{R}}

\newcommand{\co}{\mathcal{O}}

\makeatletter
\makeatletter
\newcommand*{\house}[1]{%
  \mathord{%
    \mathpalette\@house{#1}%
  }%
}
\newcommand*{\@house}[2]{%
  \dimen@=\fontdimen8 %
      \ifx#1\scriptscriptstyle\scriptscriptfont
      \else\ifx#1\scriptstyle\scriptfont
      \else\textfont\fi\fi
      3 %
  \sbox0{%
    $#1%
      \vrule width\dimen@\relax
      \overline{%
        \kern2\dimen@
        \begingroup % to keep changes of \dimen@ in #2 local
          #2%
        \endgroup
        \kern2\dimen@
      }%
      \vrule width\dimen@\relax
      \mathsurround=1.5\dimen@ % outside margin
    $%
  }%
  \ht0=\dimexpr\ht0-\dimen@\relax
  \dp0=\dimexpr\dp0+2\dimen@\relax
  \vbox{%
    \kern\dimen@ % reinsert previously removed space
    \copy0 %
  }%
}

\newcommand{\keywords}[1]{\noindent\textbf{Keywords:} #1}
\newcommand{\subjclass}[2]{\noindent\textbf{#1 Mathematics Subject Classification:} #2.}

\title{There is no 290-Theorem for higher degree forms}
\author[1]{V\'{i}t\v{e}zslav Kala}
\author[2]{Om Prakash}
\affil[1,2]{Charles University, Faculty of Mathematics and Physics, Department of Algebra, Sokolovsk\'{a} 83, 186 75 Praha 8, Czech Republic}
\affil[1]{E-mail: vitezslav.kala@matfyz.cuni.cz. ORCID: 0000-0001-5515-6801}
\affil[2]{E-mail: prakash@karlin.mff.cuni.cz. ORCID: 0009-0007-2124-9736}
\date{}

\begin{document}
\maketitle

\begin{abstract} 
We study the universality of forms of degrees greater than 2 over rings of integers of totally real number fields. We show that such universal forms always exist, but cannot be characterized by any variant of the 290-Theorem of Bhargava--Hanke.

\medskip
	 
\medskip
\subjclass{2020}{11E76, 11D25, 11D41, 11D85, 11P05, 11R80}

\medskip

\noindent\textbf{Funding:} Both authors were supported by {Czech Science Foundation} grant 21-00420M. O. P. was also supported by {Charles University} programmes PRIMUS/24/SCI/010 and UNCE/24/SCI/022.

\end{abstract}

\maketitle

\section{Introduction}

The arithmetic theory of quadratic forms has a long and colorful prehistory ranging from mathematicians in ancient Babylon, Egypt, Greece, and India, to Fermat, Lagrange, and Gauss. In the 20th century, Ramanujan, Dickson, and Willerding initiated classifications of \textit{universal} quadratic forms, i.e., positive definite forms that represent all positive integers, such as the sum of four squares $x^2+y^2+z^2+w^2$. These efforts culminated in  the 290-Theorem of Bhargava--Hanke \cite{bhargava2} saying that a positive definite quadratic form $Q$ is universal if it represents
    \begin{center}
        $ 1, 2, 3, 5, 6, 7, 10, 13, 14, 15, 17, 19, 21, 22, 23, 26,$\\
        $29, 30, 31, 34, 35, 37, 42, 58, 93, 110, 145, 203,\text{ and } 290.$
    \end{center}

This list of integers is minimal and unique in the sense that for each integer $n$ from the list, there is a form that represents all the positive integers except for $n$.

The 290-Theorem has been extended and generalized in several directions, including to subsets of positive integers (the conjectural 451-Theorem of 
Rouse \cite{rouse} concerns quadratic forms representing all odd positive integers) and to representations of quadratic forms. In both of these settings, finite \textit{criterion sets} characterizing universality exist \cite{kim-kim-oh},
but they need not be unique \cite{elkies}.

\medskip

The goal of this article is to consider universality of \textit{$m$-ic forms}, i.e., homogeneous polynomials of degree $m>2$, and the question whether they satisfy a version of the 290-Theorem.

The first example of an $m$-ic form is the sum of $m$th powers, about which Waring in his 1770 book \emph{Meditationes Algebraic\ae} asked if there is a constant $g(m)$ such that every positive integer is the sum of at most $g(m)$ $m$th powers of positive integers. Only in 1909, Hilbert established the existence of $g(m)$ for all $m\geq 1$; e.g., we have $g(1)=1, g(2)=4, g(3)=9, g(4)=19$.
After the works of many mathematicians (e.g., see \cite{kubina,linnik, niven, rieger, vaughan-wooley}), 
the values of $g(m)$ have now been almost completely determined, with $g(m)=2^m+\lfloor (3/2)^m\rfloor - 2$ for all except possibly finitely many $m$. 

Let us call an integral $m$-ic form $Q$ in $n$ variables \textit{positive definite} if $Q(x)>0$ for all $x\in\R^n, x\neq 0$ (this immediately forces $m$ to be even), and \textit{universal} if, moreover, for every positive integer $a$ there is $x\in\Z^n$ such that $Q(x)=a$. 
While Waring's problem and, more generally, representations of integers as sums of $m$th powers (e.g., the hard problems concerning the sums of two and three rational or integral cubes) received considerable attention, we are not aware of any results on other universal $m$-ic forms.

A partial reason for this is that the structural theory of $m$-ic forms is more complicated than in the case of quadratic forms, e.g., we do not have a good analogue of the Gram matrix of a quadratic form. Nevertheless, various aspects of the algebraic theory of quadratic forms were extended, starting with
Harrison \cite{harrison} who established Witt theory for higher degree forms over arbitrary fields of sufficiently large characteristics. Subsequently, results about abstract quadratic forms were extended to higher degree forms by various authors (e.g., see \cite{huang, morandi, pumpl1, pumpl2, pumpl3, pumpl4}).
Also, the negative solution to Hilbert's tenth problem by  Matiyasevich, Robinson, Davis, and Putnam (see \cite{davis})
means that higher degree Diophantine equations can be undecidable, and so one should not be surprised that studying integers represented by $m$-ic forms may be hard (and sometimes even impossible!).

Although at least some universal $m$-ic form exists (namely, the sum of $g(m)$ $m$th powers), our first main result shows that they cannot be characterized by any finite criterion set. Even more precisely, for every suitable finite subset $\mathcal{A}\subset\Z_{>0}$, there is an $m$-ic form that does not represent exactly the elements of $\mathcal A$.
(Here and in the following, we consider also the empty set to be finite.)

\begin{theorem}\label{mainoverz}
    Let $\mathcal{A}\subset\Z_{>0}$ be finite and $m\in\Z_{>2}$ even. Then the following conditions are equivalent:
    \begin{enumerate}
        \item 
    There exists a  positive definite $m$-ic form $Q$ that represents exactly $\Z_{\geq 0}\setminus \mathcal{A}$.
    \item  For all $a,b\in \Z$, we have that $ab^m\in \mathcal{A}$ implies $a\in \mathcal{A}.$
    \end{enumerate} Moreover, $Q$ can be chosen of rank $<(B+1)(g(m)+1)$, where $B$ is the largest element of $\mathcal A$.
\end{theorem}

We will prove this as Theorem \ref{th:Z main}  by explicitly constructing the form $Q$.
The fundamental difference from the case of  quadratic forms is that there are  infinitely many positive definite $m$-ic forms that represent (say) two given positive integers $a,b$, namely, $ax^m+\delta x^2y^{m-2}+by^m$ for any $\delta>0$, whereas if $a,b$ are not in the same square class, then there are only finitely many such quadratic forms $ax^2+cxy+by^2$, as the Cauchy--Schwartz inequality limits the possible range for $c\in\Z$. Thus, if one were to try to carry out an escalation process as in the proof of the 290-Theorem, in each step one would have to consider infinitely many forms, and so it should not be surprising that the whole argument may fail.

\medskip

It is very natural to extend the study of universal quadratic forms from $\Z$ to the ring of integers $\mathcal{O}_K$ in a totally real number field $K$; a quadratic form over $K$ is then called \textit{universal} if it represents all totally positive integers in $\mathcal{O}_K.$ Maa{\ss} \cite{mass} proved that the sum of three squares $x^2+y^2+z^2$ is universal over $\Q(\sqrt{5})$, but then Siegel \cite{siegel5} showed that if the sum of squares is universal over $K$, then $K$ must be $\Q$ or $\Q(\sqrt{5}).$ However, a universal quadratic form exists over every totally real number field, thanks to the result of Hsia--Kitaoka--Kneser \cite{hkk}. This raises the question of how many variables these universal forms need. Blomer--Kala \cite{blomerkala} proved that for every positive integer $M$, there are infinitely many real quadratic fields that do not admit a universal quadratic form with less than $M$ variables. Subsequently, this result was extended to cubic fields by Yatsyna \cite{yatsyna}, and then to fields of degree $d,$ where $d$ is divisible by $2$ or $3$ \cite{kala}. Regardless, finite universality criterion sets exist over all totally real number fields, even for representations of quadratic forms \cite{chan-oh}.

Similarly, Waring's problem has been extended to number fields, however, one encounters the possibility that a totally positive algebraic integer may not be expressible as the sum of $m$th powers at all. For example,  $3+\sqrt{2}\in\Z[\sqrt{2}]$ is totally positive, but cannot be written as the sum of squares, because every square $(a+b\sqrt{2})^2=a^2+2b^2+2ab\sqrt{2}$ 
has even coefficient at $\sqrt 2$.
In fact, 
Siegel \cite[Theorem III]{siegel5} established 
that, for an even integer $m>2$, the sum of $m$th powers is never universal over the ring of integers in a totally real number field $K$, unless $K=\Q$.

To overcome this issue, one considers the subring $\sum\mathcal{O}_K^m$ of $\mathcal{O}_K$ generated by $m$th powers of algebraic integers. Waring's problem then asks what is the smallest integer $G=G_K(m)$ such that for every totally positive $\alpha\in\sum\mathcal{O}_K^m$ with sufficiently large norm ($N(\alpha)>P_K(m)$), the equation 
$\alpha=x_1^m+x_2^m+\dots+x_{G}^m$ is solvable (in totally non-negative algebraic integers $x_1,\dots, x_G$)?

Siegel \cite{siegel4} solved this problem by generalizing the circle method to number fields; he obtained an upper bound $G_K(m)\leq dm\left(2^{m-1}+d\right)+1$ where $d=[K:\Q]$. Later, the work of Birch \cite[Theorem $2$]{birch} and Ramanujam \cite[Proposition 3]{ramanujam}
gave upper bounds on $G_K(m)$ independent of $d$; together, we have $G_K(m)\leq \max(8m^5, 2^m+1)$. Let us note that while $G_K(m)$ can be thus nicely bounded independently of the number field, the ``sufficiently large norm'' bound $P_K(m)$ depends on $K$ and is harder to control.

Overall, it is unclear whether universal higher degree forms over number fields even exist! 
However, we show that they do exist in Theorem \ref{p22} using Waring's problem and some estimates from geometry of numbers.

One can again consider the existence of criterion sets, except that now, by the homogeneity of an $m$-ic form $Q$, it is natural to consider only representations of elements up to multiplication by $m$th powers of units. We again show (as Theorem \ref{th:number field main}) that finite criterion sets do not exist, in a rather strong and precise sense.

\begin{theorem}\label{mainoverfields}
    Let $K$ be a totally real number field, $m>2$ an even positive integer, and $\mathcal{A}_0$ a finite subset of $\mathcal{O}_K^+$. Set $\mathcal{A}=\mathcal{A}_0\cdot\mathcal{O}_K^{\times m}=\{\delta\varepsilon^m \mid \delta\in \mathcal{A}_0, \varepsilon\in\mathcal{O}_K^{\times}\}.$ Then the following conditions are equivalent:
    \begin{enumerate}
        \item 
    There exists a totally positive definite $m$-ic form that represents exactly $\left(\mathcal{O}_K^+\setminus \mathcal{A}\right)\cup \{0\}$. 
    \item For all $\alpha,\beta\in\mathcal{O}_K$ we have that $\alpha\beta^m\in\mathcal{A}$ implies $\alpha\in\mathcal{A}.$
    \end{enumerate}
\end{theorem} 

This theorem includes Theorem \ref{mainoverz} as a special case. However, the separate proof of Theorem \ref{mainoverz} is easier, and yields an explicit bound on the rank of the form $Q$, and so we include both in the paper.

Finally, note that \textit{infinite} criterion sets always exist, as the trivial example of all of $\mathcal{O}_K^+$ shows. Applying Theorem \ref{mainoverfields} to $\mathcal A_0=\{\alpha\}$ for an element $\alpha\in\co_K^+$ that is $m$th powerfree (in the sense that  if $\beta^m\mid\alpha$, then $\beta\in\mathcal{O}_K^{\times}$), we see that such $\alpha$ must be contained in every criterion set. Conversely, the set of all $m$th powerfree elements is clearly a criterion set, and so this is the unique minimal criterion set with respect to inclusion. This set has positive density (say, when we order classes of totally positive integers modulo units by norm), and so every criterion set for $m$-ic forms must also have positive density.

To conclude, let us mention that our results leave very much open the extensions of other results on universal quadratic forms to the $m$-ic setting (although the second author \cite{prakash} very recently investigated some of them). In particular, we do not know much about the minimal ranks of universal $m$-ic forms over number fields -- but expect that this will turn out to be a fruitful direction of future research!

\section*{Acknowledgments}

We thank Pavlo Yatsyna and the UFOCLAN group for interesting and helpful discussions about the paper, and the anonymous referees for their extremely fast and detailed reviews that helped us improve the paper.

\section{Preliminaries}\label{s4}
Let $K$ be a \emph{totally real} number field of degree $d$, i.e.,  all its embeddings $\sigma_1, \sigma_2,\dots, \sigma_d:K\rightarrow \mathbb{C}$ have image in $\mathbb{R}.$ 

We denote its \emph{ring of integers} and group of units by $\mathcal{O}_K$ and $\mathcal{O}_K^{\times}$. An element $\alpha\in K$ is  \emph{totally positive} if $\sigma_i(\alpha)>0$ for all $1\leq i\leq d.$ For a subset $S\subset K$, we denote  by $S^+$ the set of all  totally positive elements of $S$. We write $\alpha\succ \beta$ if $\alpha-\beta$ is totally positive. 

For $\alpha\in K,$ its \emph{norm} and \emph{trace}  are $N(\alpha)=\sigma_1(\alpha)\sigma_2(\alpha)\cdots\sigma_d(\alpha)$ and $Tr(\alpha)=\sigma_1(\alpha)+\sigma_2(\alpha)+\dots+\sigma_d(\alpha).$ Also, we define the \emph{house} $\house{\alpha}$ of $\alpha$ as $\house{\alpha}=\max_{1\leq i\leq d} |\sigma_i(\alpha)|.$

Let us consider the Minkowski embedding $\sigma:K\rightarrow \mathbb{R}^d$ given by $\sigma(\alpha)=(\sigma_1(\alpha), \dots,\sigma_d(\alpha)).$
When $\alpha\in\mathcal{O}_K^+,$ then $\sigma(\alpha)$ lies in the totally positive orthant $\mathbb{R}^{d,+}=\{(x_1,x_2,\dots,x_d)\in\mathbb{R}^d \mid x_i>0 \text{ for all } 1\leq i\leq d\}$ in $\mathbb{R}^d.$

We now define $m$-ic forms over $K.$
\begin{definition}
    Let $m$ be a positive integer. An \emph{$m$-ic form} over $K$ is 
    $$Q\left(x_1, x_2, \dots, x_n\right)=\underset{i_1+i_2+\dots+i_n=m}{\underset{i_1,\dots, i_n \geq 0}{\sum}} a_{i_1i_2\dots i_n}x_1^{i_1}x_2^{i_2}\cdots x_n^{i_n}, $$ where $a_{i_1i_2\dots i_n}\in \mathcal{O}_K.$

    An $m$-ic form $Q$ over $K$ is \emph{totally positive definite} if for all embeddings $\sigma:K\rightarrow\R$ we have that $$\sigma(Q)\left(x_1, x_2, \dots, x_n\right)=\underset{i_1+i_2+\dots+i_n=m}{\underset{i_1,\dots, i_n \geq 0}{\sum}} \sigma(a_{i_1i_2\dots i_n})x_1^{i_1}x_2^{i_2}\cdots x_n^{i_n}$$ is positive definite over $\mathbb{R}$ in the sense that $\sigma(Q)(x)>0$ for all $x\in\R^n\setminus\{(0,\dots,0)\}$.
\end{definition}

Observe that the degree $m$ of a totally positive definite form $Q$ must be even, as we have $Q\left(-x_1, -x_2, \dots, -x_n\right)=\left(-1\right)^mQ\left(x_1, x_2, \dots, x_n\right).$

We say that $\alpha\in \co_K$ is \textit{represented} by an $m$-ic form $Q,$ if there is $\left(y_1, y_2,\dots, y_n\right)\in\mathcal{O}_K^n$ such that $Q\left(y_1, y_2,\dots, y_n\right)=\alpha.$ A totally positive definite $m$-ic form $Q$ is \emph{universal} if $Q$ represents all elements of $\mathcal{O}_K^+.$

Let $Q_1$ and $Q_2$ be $m$-ic forms over $K$ in $n_1$ and $n_2$ variables. The \emph{orthogonal sum} of $Q_1$ and $Q_2$ is the $m$-ic form $Q=Q_1\perp Q_2$ given by 
$Q(x_1,\dots,x_{n_1},y_1,\dots,y_{n_2})=Q_1(x_1,\dots,x_{n_1})+Q_2(y_1,\dots,y_{n_2})$.

\section{Main result over $\mathbb{Z}$}\label{s3}

Let us begin the case $K=\Q$  by showing that one cannot generalize 290-Theorem for higher degree forms. Throughout this section, $m>2$ is an even integer.

First, recall Hilbert's theorem regarding Waring's problem.
\begin{theorem}[\cite{hilbert}]\label{Waringproblem1}
    For each fixed $m,$ there exists $g(m)<\infty$ such that every positive integer can be expressed as the sum of at most $g(m)$ $m$th powers.
\end{theorem}

Moreover, it is conjectured \cite{vaughan-wooley} that $g(m)=2^m+\lfloor(3/2)^m\rfloor-2$ for every $m$. This has been verified for $m\leq 471,600,000$ by Kubina--Wunderlich \cite{kubina}. Unconditionally it is known that $g(m)\leq 2^m+\lfloor(3/2)^m\rfloor+\lfloor(4/3)^m\rfloor-2$ (cf. \cite[bottom of page 1]{vaughan-wooley}), from which one can easily get that $g(m)<2^{m+1}$ if one would like to have a concise upper bound.

\begin{proposition}\label{15h}
    Given a positive integer $B$ such that $B+1$ is $m$th powerfree, there is a positive definite $m$-ic form $Q$ that represents all the positive integers $\leq B$ but not $B+1.$
\end{proposition}

 \begin{proof}    
    For any $\delta\geq B$ consider the form 
    \begin{equation*}
        Q\left(x_1, x_2, \dots, x_{B}\right)=\sum_{i=1}^B ix_i^{m} + \delta \sum_{1\leq i<j \leq B} x_i^2x_{j}^{m-2},
    \end{equation*}
    which is positive definite, because each variable has an even exponent and each coefficient is positive.

    When we plug-in $x_i=1$ and $x_j=0$ for all $j\neq i$, we get $Q(0,\dots,0,1,0,\dots,0)=i$.
    Hence $Q$ represents every positive integer $\leq B.$ 
    
    Now we claim that $Q$ does not represent $B+1.$
    For contradiction, suppose that there exists $\left(y_1, y_2, \dots, y_{B}\right)\in \Z^B$ such that $Q\left(y_1, y_2, \dots, y_{B}\right)=B+1.$ If there exist $i<j$ such that $(y_i,y_j)\neq (0,0)$, then we have 
    $$Q\left(y_1, y_2, \dots, y_{B}\right)\geq iy_i^m+jy_j^m+\delta y_i^2y_j^{m-2}\geq i+j+\delta> B+1.$$  
    Thus there exists a unique $i$ such that $y_i$ is non-zero, and we have $Q\left(0, 0, \dots, y_{i},\dots, 0\right)=iy_i^{m}=B+1,$ which is impossible, since $i\leq B$ and $B+1$ is $m$th powerfree.
    \end{proof}

To establish our main result over $\Z$, we need the following proposition about representation of large elements.
\begin{proposition}\label{largeelement}
    Let  $B$ be a positive integer. The positive definite $m$-ic form 
    \begin{equation*}
        Q_B(x_{01}, \dots, x_{0g(m)},\dots, x_{B1},\dots, x_{Bg(m)})=\sum_{j=0}^B (B+1+j)\sum_{i=1}^{g(m)}x_{ji}^m
    \end{equation*}
represents exactly $\mathbb{Z}_{\geq 0}\setminus \left\{1, 2, \dots, B\right\}.$ Moreover, $\text{rank}(Q_B)=(B+1)g(m)$.
\end{proposition}

\begin{proof} $Q_B$ is the sum of non-negative terms, and so it is positive definite. Further, as all its coefficients are greater than $B$, $Q_B$ does not represent any integer $\leq B$. The rank of $Q_B$ is also clearly $(B+1)g(m)$.

Let us now take an integer $n\geq B+1$.
First, it is easy to observe that $n$  can be written as $$n=(B+1)y_0+(B+2)y_1+\dots+(2B+1)y_{B} \text{ for some }y_j\in\Z_{\geq 0}.$$ To see this, let $j=0,\dots, B$ be such that $n\equiv j\pmod {B+1}$, and set 
\begin{itemize}
    \item $y_h=0$ for all $h\neq 0, j$,
    \item $y_j=1$, $y_0=(n-j)/(B+1)-1$ if $j\neq 0$,
    \item $y_0=n/(B+1)$ if $j=0$.
\end{itemize}
(This observation can be viewed as an easy result on the Frobenius coin problem.)

Further, by Theorem \ref{Waringproblem1} each $y_j\in\Z_{\geq 0}$ is represented by the $m$-ic form $\sum_{i=1}^{g(m)}x_{ji}^m.$ Thus $Q_B$ represents $n$ as desired.
\end{proof}

\begin{theorem}\label{th:Z main}
    Let $\mathcal{A}\subset\Z_{>0}$ be finite and $m\in\Z_{>2}$ even. Then the following conditions are equivalent:
    \begin{enumerate}
        \item 
    There exists a  positive definite $m$-ic form $Q$ that represents exactly $\Z_{\geq 0}\setminus \mathcal{A}$.
    \item  For all $a,b\in \Z$, we have that $ab^m\in \mathcal{A}$ implies $a\in \mathcal{A}.$
    \end{enumerate} Moreover, $Q$ can be chosen of rank $<(B+1)(g(m)+1)$, where $B$ is the largest element of $\mathcal A$.
\end{theorem}

\begin{proof} 
$( (1) \Rightarrow (2) )$ Assume that $Q$ represents exactly $\Z_{\geq 0}\setminus \mathcal{A}$, and take $a\notin \mathcal{A}$. Thus we have $Q(x)=a$ for some $x\in\Z^n$ and then $Q(bx)=ab^m$ for every $b\in\Z$ by homogeneity of $Q$. As $Q$ represents $ab^m$, we must have $ab^m\notin\mathcal A$, as desired.

$( (2) \Rightarrow (1) )$ Assume that for all $a,b\in\Z$ we have that $ab^m\in\mathcal{A}$ implies $a\in \mathcal{A}$. Let $B$ be the largest element of $\mathcal A$ and let $Q_B$ be the form defined in Proposition \ref{largeelement}. Hence $Q_B$ represents exactly $\mathbb{Z}_{\geq 0}\setminus \left\{1, 2, \dots, B\right\}$ and has rank $(B+1)g(m).$

To construct the desired form $Q$, we need to arrange the representation of small elements that do not lie in $\mathcal A$. To do that, denote by $\{b_1,\dots,b_k\}$ the complement of $\mathcal{A}$ in $\{1,2,\dots,B\}$, i.e., $\{b_1,\dots,b_k\}=\{1,2,\dots,B\}\setminus \mathcal A$ with $k\leq B$, and define
\begin{equation*}
    Q'(y_1,\dots,y_{k})=\sum_{i=1}^k b_i y_i^m +\sum_{1\leq i<j\leq k} \delta y_i^2 y_j^{m-2}
\end{equation*}
for some $\delta>B.$ Now, consider the form 
\begin{equation*}
    Q= Q_B \perp Q'.
\end{equation*}
Clearly, $Q$ is positive definite and has rank $(B+1)g(m)+k\leq (B+1)g(m)+B<(B+1)(g(m)+1)$. 

It is clear that $Q$ represents $\Z_{\geq 0}\setminus \mathcal{A},$ because $Q_B$ represents all integers $\geq B$ and $Q'$ represents all the integers $\leq B$ which are not in $\mathcal{A}.$

Suppose now for contradiction that $Q$ represents some $a\in\mathcal{A}$.
As $a$ is represented by the orthogonal sum $Q= Q_B \perp Q'$, we have $B\geq a=u+v$ where $u$ is represented by $Q_B$ and $v$ is represented by $Q'$. However, the only integer $u\leq B$ that is represented by $Q_B$ is $u=0$, and so $a=v$ is represented by $Q'$. 

Now we proceed as in the proof of Proposition \ref{15h}:
$Q'(y_1,\dots,y_k)=a$ for some  $y_1,\dots,y_k\in \Z$. If there exist $i<j$ such that $(y_i,y_j)\neq (0,0)$, then $B\geq a=Q'(y_1,\dots,y_k)>\delta y_i^2 y_j^{m-2}>B,$ which is impossible.

Hence there exists a unique $i$ such that $y_i\neq 0$. But then we have
$$b_i y_i^m=Q'(0,0,\dots,y_i, \dots 0)=a\in \mathcal{A}.$$
Thus by the assumption, we have $b_i\in \mathcal{A},$ which contradicts the choice of $b_i\in\{1,2,\dots,B\}\setminus \mathcal A.$ 
\end{proof}

\section{Geometry of numbers estimates}\label{s5}
We 
begin this section by recalling several facts about sums of $m$th powers and Waring's problem that we then use in several lemmas that help us to represent totally positive algebraic integers with large norms. Towards the end, we establish the existence of universal higher degree forms.

Throughout this section, we fix a totally real number field $K$ of degree $d$ and an even positive integer $m$.

We will use the following result concerning Waring's problem over number fields (see \cite[Page $137,$ Paragraph $2$]{ramanujam}). Recall that $\sum\mathcal{O}_K^m$ denotes the subring of $\mathcal{O}_K$ generated by $m$th powers of elements of $\mathcal{O}_K$.

\begin{theorem} \label{t12}
   Let $K$ be a totally real number field. There exist constants $P=P_K(m)$ and $G=G_K(m)\leq \max (8m^5, 2^m+1)$ such that every totally positive $\alpha\in\sum\co_K^m$ with $N(\alpha)>P$ can be written as the sum of at most $G$ $m$th powers of totally positive integers in $K.$
\end{theorem}

Note that $\sum\mathcal{O}_K^m$ has finite index as an additive subgroup in $\mathcal{O}_K.$ To see this, observe that $m!\mathcal{O}_K\subset\sum\mathcal{O}_K^m\subset\mathcal{O}_K$ thanks to the identity
$m!x = \sum_{k=0}^{m-1}(-1)^{m-1-k}{n-1\choose k}\left((x+k)^m-k^m\right)$, see, e.g., \cite[Page $142$]{bateman}. 
So there is a finite set of representatives of classes of $\mathcal{O}_K$ modulo $\sum\mathcal{O}_K^m$. Further, we can add elements of $\Z_{\geq 0}\subset\sum\mathcal{O}_K^m$ to the representatives to assume that all the representatives are totally positive.

\begin{convention}
For the rest of article, let us fix (additive group) representatives $\theta_1, \theta_2, \dots, \theta_r\in\co_K^+$ for $\mathcal{O}_K /\sum\mathcal{O}_K^m$.
\end{convention}

Consider the subgroup $$\mathcal{O}_K^{\times m}=\left\{\varepsilon^m \mid \varepsilon\in\mathcal{O}_K^{\times} \right\}\subset \mathcal{O}_K^{\times, +}.$$ 
Let $\mathcal{G}$ be a \textit{compact} fundamental domain for the action of $\sigma(\mathcal{O}_K^{\times m})$ on  $$U^+= \left\{(x_1, x_2,\dots, x_d)\in\mathbb{R}^{d,+} \mid \prod_{i=1}^d x_i=1\right\},$$ i.e., $\mathcal{G}$ is a compact subset of $\mathbb{R}^{d,+}$ such that for each $x\in U^+$ there is $\varepsilon^m\in \mathcal{O}_K^{\times m}$ with $\sigma(\varepsilon^m)x\in\mathcal G$.

Further, let $\mathcal{F}=\R^+\cdot \mathcal{G}$. This is a fundamental domain for the action of $\sigma(\mathcal{O}_K^{\times m})$ on the totally positive orthant $\mathbb{R}^{d,+}$, i.e., 
for each $x\in \mathbb{R}^{d,+}$ there is $\varepsilon^m\in \mathcal{O}_K^{\times m}$ with $\sigma(\varepsilon^m)x\in\mathcal F$.

In fact, $\mathcal{F}$ and $\mathcal G$ can be described quite explicitly by  Shintani's unit theorem \cite[Theorem $9.3$]{neukirch}, but we will not need this.

\begin{convention}
For the rest of article, we fix some fundamental domains $\mathcal{F}$ and $\mathcal G$ as above.
\end{convention}

\begin{lemma}\label{l18}\label{l17}
There is a constant $0<c<1$ with the following property: 

For every $\beta\in\mathcal{O}_K^+$ there exists $\varepsilon\in\mathcal{O}_K ^{\times}$ such that

a) $\beta \varepsilon^m\succ \lfloor c N(\beta)^{1/d}\rfloor$, and 
   
b) $N(\beta\varepsilon^m - n_{X})>c^d N(\beta)$ if $N(\beta)>X$, where $X\in\R^+$ and $n_X=\lfloor c X^{1/d}\rfloor$.
\end{lemma}

\begin{proof}
First of all, let us prove that there is a constant $0<c<1$ such that for all $(x_1,\dots,x_d)\in\mathcal{F},$ we have 
    \begin{equation}\label{eq:F}
        x_i>2c(x_1\cdots x_d)^{1/d}\text{ for all }1\leq i\leq d.
    \end{equation}

To establish this, consider the map ${\bf \pi} : \mathcal{G} \rightarrow \mathbb{R}^{+},  (g_1, g_2,\dots, g_d) \mapsto {\min}(g_1,\dots,g_d).$
   Since $\mathcal{G}$ is compact and $\pi$ is continuous, it has a minimum $\ell=\min\{\pi(g) \mid g\in\mathcal{G}\}\in\R^+$.

Let now $(x_1,\dots,x_d)\in\mathcal{F}$ and denote $N=(x_1\cdots x_d)^{1/d}$.
   Then $(x_1,\dots,x_d)=(g_1,\dots,g_d)N$ for some $(g_1,\dots,g_d)\in\mathcal{G}.$ But then it follows that $x_i=g_i N\geq \ell N$, and so \eqref{eq:F} holds for any $2c<\ell$. As $\ell$ is positive, we can choose $c$ sufficiently small so that $0<c<1$ as desired.

   \medskip

By the definition of $\mathcal F$, for every $\beta\in\mathcal{O}_K^+$ there is an $\varepsilon\in\mathcal{O}_K^{\times}$ such that $\sigma(\beta\varepsilon^m)\in\mathcal{F}.$ From \eqref{eq:F} it follows that 
    \begin{equation*}\label{e1}
        \sigma_i(\beta\varepsilon^m)>2cN(\beta)^{1/d}\geq \lfloor c N(\beta)^{1/d}\rfloor
    \end{equation*}
    for all $1\leq i \leq d.$ This immediately implies part a), $\beta \varepsilon^m\succ \lfloor c N(\beta)^{1/d}\rfloor$.
    
   \medskip

    Assume now $N(\beta)>X$ and let $n_X=\lfloor cX^{1/d}\rfloor$. We have
    \begin{equation*}
        \sigma_i(\beta\varepsilon^m) - n_X>2cN(\beta)^{1/d}-c N(\beta)^{1/d}=cN(\beta)^{1/d}.
    \end{equation*}
    Consequently, we have 
   \begin{equation*}
   N(\beta\varepsilon^m -n_X)=\prod_{i=1}^{d}\left(\sigma_i(\beta\varepsilon^m )-n_X\right) > c^dN(\beta).\qedhere
   \end{equation*}
\end{proof}

Recall that in the following lemma, the elements $\theta_1, \theta_2, \dots, \theta_r\in\co_K^+$ denote the representatives for $\mathcal{O}_K /\sum\mathcal{O}_K^m$, as introduced above.

\begin{lemma}\label{l20}\label{l19}
Let $n\in\Z_{\geq 0}$. There is a positive integer $M=M_n\geq n^d$ such that for every $\beta\in \mathcal{O}_K^+$ with $N(\beta)>M,$
      there exist $\varepsilon\in\mathcal{O}_K^{\times}$ and $j\in\{1,2,\dots,r\}$ such that 
      $\beta\varepsilon^m=n+\theta_j+\gamma$ for some $0\prec \gamma\in\sum\mathcal{O}_K^m$ that satisfies $N(\gamma)>c^dN(\beta)$, where $c>0$ is the constant from Lemma $\ref{l17}$.
\end{lemma}

\begin{proof}
Let $M\in\Z_{\geq n^d}$ be such that $n_M=\lfloor c M^{1/d}\rfloor>n+\max_{1\leq i\leq r}{\house{\theta_i}}$, where $c$ is the constant from Lemma \ref{l17}. Note that $n_M\succ n+\theta_i\succ 0$ for all $i$.

By Lemma \ref{l18} a), for every $\beta\in\mathcal{O}_K^+$ with $N(\beta)>M$ there exists $\varepsilon\in\mathcal{O}_K^{\times}$ such that 
\begin{equation}\label{eq:2}
\beta\varepsilon^m\succ\lfloor c N(\beta)^{1/d}\rfloor\succeq \left\lfloor c M^{1/d}\right\rfloor\succ n+\theta_i    
\end{equation} for all $i$.

As $\theta_1, \theta_2, \dots, \theta_r\in\co_K^+$ form the representatives for $\mathcal{O}_K /\sum\mathcal{O}_K^m$, there is $j$ such that $\beta\varepsilon^m-\theta_j\in\sum\mathcal{O}_K^m$. 
As $n=1+\dots+1\in\sum\mathcal{O}_K^m$, we also have $\gamma=\beta\varepsilon^m-\theta_j-n\in\sum\mathcal{O}_K^m$.
By \eqref{eq:2}, we have $\gamma\succ 0$.

Finally, using $n_M\succ n+\theta_i$ for all $i$ and  Lemma \ref{l18} b), we have
\begin{align*}
N(\gamma)=N(\beta\varepsilon^m-\theta_j-n)=&\prod_{i=1}^d \left(\sigma_i(\beta\varepsilon^m)-\sigma_i(\theta_j+n)\right)\\ 
>&\prod_{i=1}^d \left(\sigma_i(\beta\varepsilon^m)-n_M\right)=
N(\beta\varepsilon^m - n_{M})>c^d N(\beta),    
\end{align*}
as desired.
\end{proof}

\begin{lemma}\label{l21}
    There exists a positive integer $L\geq M_0$ such that the totally positive definite $m$-ic form $$Q_1\left(x_1,\dots,x_r,y_1,\dots,y_{G}\right)=\sum_{i=1}^r\theta_ix_i^m + \sum_{i=1}^{G}y_i^m$$ represents every $\beta\in\mathcal{O}_K^+$ with $N(\beta)>L$, where $M_0$ is the integer from Lemma $\ref{l20}$ for $n=0$.
\end{lemma}

\begin{proof}
    Since all $\theta_i$s are totally positive and $m$ is even,  $Q_1$ is totally positive definite. 
Set $L=\lceil\max(M_0, c^{-d}P)\rceil,$ where $c$ and $P$ are from Lemma \ref{l18} and Theorem \ref{t12}.     

    For any $\beta\in\mathcal{O}_K^+$ with $N(\beta)>L$, by Lemma \ref{l20} with $n=0$, there exist $\varepsilon\in\mathcal{O}_K^{\times}$ and $j\in\{1,2,\dots,r\}$ such that $\beta\varepsilon^m=\theta_j+\gamma$ with $0\prec\gamma\in \sum\mathcal{O}_K^m.$ 
    We then have $\beta=\theta_j\varepsilon^{-m}+\gamma\varepsilon^{-m}$ and $\gamma\varepsilon^{-m}\succ 0$.
    
    By Lemma \ref{l20} we have $N(\gamma\varepsilon^{-m})>c^dN(\beta)>c^dL\geq P.$ Now, it follows from Theorem \ref{t12} that $\gamma\varepsilon^{-m}$ is represented by the sum of $m$th powers $\sum_{i=1}^{G}y_i^m.$ Thus, by putting $x_k=\varepsilon^{-1}$ and $x_i=0$ for all $i\neq k$, we see that $Q_1$ represents $\beta$.
\end{proof}

We are now prepared to prove the existence of a universal $m$-ic form.

\begin{theorem}\label{p22}
    Given a totally real number field $K$ and an even positive integer $m>2,$ there exists a universal $m$-ic form over $K$. 
\end{theorem}

\begin{proof}
   Let $L$ be the positive integer and $Q_1$ the $m$-ic form from Lemma \ref{l21}. There are only finitely many totally positive integers with norm $\leq L$, up to multiplication by elements of $\mathcal{O}_K^{\times m}.$ So, let us fix a set of representatives $\{\alpha_1,\alpha_2,\dots,\alpha_s\}$ for them and consider the $m$-ic form $$q(z_1,z_2,\dots,z_s)=\sum_{i=1}^s\alpha_i z_i^m.$$ 
    
    We will prove that the $m$-ic form 
    $$Q=Q_1 \perp q$$
    is universal. 
    Since $Q_1$ and $q$ are totally positive definite, it follows that $Q$ is also totally positive definite, and so it suffices to prove that $Q$
    represents every element of $\mathcal{O}_K^+.$
    
 Lemma \ref{l21} ensures that $Q$ represents all totally positive integers with norm $>L.$

    If $\beta\in\mathcal{O}_K^+$ with $N(\beta)\leq L,$ then $\beta=\alpha_j\varepsilon^m$ for some $j\in\{1,2,\dots,s\}$ and some $\varepsilon\in\mathcal{O}_K^{\times}.$ By setting $z_j=\varepsilon$ and $z_i=0$ for all $i\neq j$, we see that $\beta$ is represented by $q$, and thus also by $Q.$
\end{proof}

\section{Construction of $m$-ic forms over totally real number fields} \label{s6}

Using Theorem \ref{p22} we now establish that there is a form that represents all totally positive integers of sufficiently large norm. 

\begin{proposition}\label{p23}
    Fix a positive integer $L$ from Lemma $\ref{l21}$. For each $B>L,$ there exists a totally positive definite, $m$-ic form $Q_B$ that represents all elements of $\mathcal{O}_K^+$ of norm $>B$ and does not represent any element of norm $\leq B.$
\end{proposition}

\begin{proof}
Let $B,n$ be positive integers such that $n^d>B>L$. Let
$M=M_n\geq n^d$ be the integer from Lemma $\ref{l20}$ used for $n$.
Let $C=\max(M, c^{-d}B)$, where $0<c<1$ is the constant from Lemma $\ref{l17}$. Note that we have $C>B>L$.

    Let $\mathcal{S}$ be a (finite) set of representatives for $\alpha\in\mathcal{O}_K^+$ with $B<N(\alpha)\leq C$ modulo $\mathcal{O}_K^{\times m}$, and consider the $m$-ic form 
    \begin{equation*}
    Q_B=\underset{\alpha\in\mathcal{S}}{{\perp}}\alpha Q,
     \end{equation*}
     where $Q$ is the universal form from Theorem \ref{p22}. Since $Q$ is totally positive definite and $\alpha\in\mathcal{O}_K^+$, it follows that $Q_B$ is totally positive definite. 

     Let us prove that $Q_B$ represents exactly the elements of norm $>B$.
     
     If $\beta\in\mathcal{O}_K^+$ is an element of norm $\leq B$, 
     then $\beta$ is not the sum of totally positive integers of norm $>B$ 
     by the easy observation that if $\alpha\succ\beta\succ 0$, then $N(\alpha)>N(\beta)$      
     (for a more precise result, see \cite[Lemma $3.1$]{om1}).
     Thus $Q_B$ does not represent $\beta$. 

     Let $\beta\in\mathcal{O}_K^+$ be an element with $B<N(\beta)\leq C.$ Then we can write $\beta=\alpha\varepsilon^m$ for some $\alpha\in\mathcal{S}$ and some $\varepsilon\in\mathcal{O}_K^{\times}.$ Since $Q$ is universal, it represents the totally positive unit $\varepsilon^m.$ Hence $\beta$ is represented by $Q_B.$

    By Lemma \ref{l20}, for every $\beta\in\mathcal{O}_K^+$ with $N(\beta)>C,$ there is $\varepsilon\in\mathcal{O}_K^{\times}$ such that $$\beta\varepsilon^m=n+\theta_j+\gamma,$$ for some $j\in\{1,2,\dots, r\}$ and some $0\prec\gamma\in\sum\mathcal{O}_K^m.$ Furthermore, we have $N(\gamma)>c^dN(\beta)>c^dC\geq B$, and so also $N(\theta_j+\gamma)>B$.

    Letting $\varepsilon_1=\varepsilon^{-1}$ and $\beta_1=\varepsilon^{-m}(\theta_j+\gamma)$, we see that 
    $$\beta=n\varepsilon_1^m+\beta_1\text{ with }  B<N(\beta_1)<N(\beta).$$

    If $\beta_1$ still satisfies $N(\beta_1)>C,$
    then we can further decompose it as above, $\beta_1=n\varepsilon_2^m+\beta_2$, and then eventually $\beta_{i-1}=n\varepsilon_i^m+\beta_i$. As the norms of the elements $\beta_i$ are decreasing, eventually we obtain
    $$\beta=n(\varepsilon_1^m+\dots+\varepsilon_k^m)+\beta_k\text{ with } B< N(\beta_k)\leq C.$$

    As we also have $B< N(n)\leq C$, there are some representatives $\alpha\in\mathcal S$  that lie in the same $\mathcal{O}_K^{\times m}$-classes as the elements $n$ and $\beta_k$.  
    If necessary (i.e., if $n$ and $\beta_k$ have the same representative), we can group the summands together to write
    $$\beta=\sum_{\alpha\in\mathcal S}\alpha\delta_\alpha$$
    for some $\delta_\alpha\in\co_K^+\cup\{0\}$ (we have $\delta_\alpha\neq 0$ only for the 1 or 2 values of $\alpha\in\mathcal S$ that lie in the same classes as $n, \beta_k$).

    Each element $\delta_\alpha$ is represented by the universal form $Q$, and so $\beta$ is represented by $Q_B$ as desired.
\end{proof}

Now we are finally ready to prove the number field analogue of Theorem \ref{mainoverz}. 

\begin{theorem}\label{th:number field main}
    Let $K$ be a totally real number field, $m>2$ an even positive integer, and $\mathcal{A}_0$ a finite subset of $\mathcal{O}_K^+$. Set $\mathcal{A}=\mathcal{A}_0\cdot\mathcal{O}_K^{\times m}=\{\delta\varepsilon^m \mid \delta\in \mathcal{A}_0, \varepsilon\in\mathcal{O}_K^{\times}\}.$ Then the following conditions are equivalent:
    \begin{enumerate}
        \item 
    There exists a totally positive definite $m$-ic form that represents exactly $\left(\mathcal{O}_K^+\setminus \mathcal{A}\right)\cup\{0\}$. 
    \item For all $\alpha,\beta\in\mathcal{O}_K$ we have that $\alpha\beta^m\in\mathcal{A}$ implies $\alpha\in\mathcal{A}.$
    \end{enumerate}
\end{theorem}  

\begin{proof}
    $( (1) \Rightarrow (2) )$ The proof of this implication is the same as in Theorem \ref{th:Z main}.

     $( (2) \Rightarrow (1) )$ Assume that for all $\alpha,\beta\in \mathcal{O}_K,$ we have that $\alpha\beta^m\in\mathcal{A}$ implies $\alpha\in\mathcal{A}.$ 

 Fix a positive integer $L$ from Lemma \ref{l21} and let $B=\max\left(L+1, \max_{\alpha\in\mathcal{A}}(N(\alpha))\right)$.
 Let $Q_B$ be a form from Proposition \ref{p23} that represents all elements of $\mathcal{O}_K^+$ of norm $>B$ and does not represent any element of norm $\leq B.$

 Let $\mathcal S$ be a (finite) set of representatives of classes of elements $\alpha\in\mathcal{O}_K^+\setminus \mathcal{A}$, $N(\alpha)\leq B$, up to multiplication by elements of 
 $\mathcal{O}_K^{\times m}$. Note that by the assumption on $\mathcal A$, for every $\varepsilon^m\in \mathcal{O}_K^{\times m}$ we have $\alpha\in\mathcal{O}_K^+\setminus \mathcal{A}$ if and only if $\alpha\varepsilon^m\in\mathcal{O}_K^+\setminus \mathcal{A}$, and so $\mathcal S$ is well-defined.

 Finally, fix some $\mu\in\co_K^+$ with $N(\mu)>B$ and let
 $$Q=Q_B\perp q,\text{ where } q=\sum_{\alpha\in \mathcal S}\alpha x_\alpha^m+\mu\sum_{\alpha\neq\beta\in \mathcal S} x_\alpha^2x_\beta^{m-2}.$$

The form $Q$ is clearly totally positive definite and $m$-ic, and so we need to show that it represents exactly $\left(\mathcal{O}_K^+\setminus \mathcal{A}\right)\cup\{0\}$.
 
\medskip

Consider $\gamma\in\mathcal A$ and assume for contradiction that $Q$ represents $\gamma$. We have $N(\gamma)\leq B$ and as $Q_B$ does not represent any element of norm $\leq B$, $\gamma$ is represented by $q$, i.e., $\gamma=q(x_\alpha\mid \alpha\in\mathcal S)$ for some $x_\alpha\in\co_K$.

If $x_\alpha\neq 0, x_\beta\neq 0$ for some $\alpha\neq\beta$, then 
$$\gamma=\sum_{\alpha\in \mathcal S}\alpha x_\alpha^m+\mu\sum_{\alpha\neq\beta\in \mathcal S}x_\alpha^2x_\beta^{m-2}\succ\mu x_\alpha^2x_\beta^{m-2}.$$
But then $B\geq N(\gamma)>N(\mu x_\alpha^2x_\beta^{m-2})\geq N(\mu)>B,$
a contradiction.

Thus exactly one $x_\alpha\neq 0$ and $x_\beta= 0$ for all $\beta\neq\alpha$. But then $\mathcal A\ni \gamma=q(x_\alpha\mid \alpha\in\mathcal S)=\alpha x_\alpha^m$. By the assumption on $\mathcal A$, this implies that $\alpha\in\mathcal A$, a contradiction with the choice of $\mathcal S$.

\medskip

Now let $\gamma\not\in\mathcal A,\gamma\neq 0$ ($0$ is represented trivially). If $N(\gamma)\leq B$, then $\gamma=\alpha\varepsilon^m$ for some $\alpha\in\mathcal S$ and $\varepsilon^m\in \mathcal{O}_K^{\times m}$. Thus $\alpha x_\alpha^m$ represents $\gamma$, which in turn implies that $q$ and $Q$ also represent $\gamma$.

If $N(\gamma)> B$, then $\gamma$ is represented by $Q_B$, and thus also by $Q$.
\end{proof}

\providecommand{\bysame}{\leavevmode\hbox to3em{\hrulefill}\thinspace}
\providecommand{\MR}{\relax\ifhmode\unskip\space\fi MR }
\providecommand{\MRhref}[2]{%
	\href{http://www.ams.org/mathscinet-getitem?mr=#1}{#2}
}
\providecommand{\href}[2]{#2}

\end{document}